\documentclass[11pt]{amsart}

\numberwithin{equation}{section}
\usepackage{mathrsfs,txfonts}
\usepackage[dvips]{graphicx}
\usepackage{latexsym,bm}
\usepackage{amsfonts}
\usepackage{indentfirst}
\usepackage{amsthm}
\usepackage{cases}
\usepackage{pifont}
\usepackage{amsmath}
\usepackage{amssymb}

\allowdisplaybreaks
\newtheorem{lem}{\quad\textbf{\Large Lemma}}[section]
\newtheorem{thm}[lem]{\quad\textbf{\Large Theorem}}
\newtheorem{cor}[lem]{\quad\textbf{\Large Corollary}}
\newtheorem{prop}[lem]{\quad\textbf{\Large Propsotion}}

\def\squarebox#1{\hbox to #1{\hfill\vbox to #1{\vfill}}}

\usepackage{enumerate}
\begin{document}
\title[Generalizations of a terminating summation formula]
{Generalizations of a terminating summation formula of basic
hypergeometric series  and their applications}
\author{Jun-Ming Zhu  }
\address{Department of Mathematics, Luoyang Normal University,
Luoyang City, Henan Province 471022, China}
\email{junming\_zhu@163.com}
\thanks{
 \\\indent This research is partly supported by the
 National Science Foundation of China
 (Grant No. 11371184).
 }

\begin{abstract}
We generalize a terminating summation formula to a unilateral
nonterminating, and further, a bilateral  summation formula by a
property of analytic functions. The unilateral one is proved to be a
$q$-analogue of a $_4F_3$-summation formula. And, an identity
unifying Jacobi's triple product identity and the quintuple product
identity is obtained as a special case of the bilateral one.
\end{abstract}
 \maketitle

\section{Introduction}
In this papar, we  suppose $0<|q|<1$ and follow the notations and
terminology in \cite{gasper}. The $q$-shifted factorials are defined
respectively by
\begin{equation*}\label{def}
(a;q)_\infty=\prod_{i=0}^{\infty}(1-aq^{i}) \qquad \mbox{and}\qquad
(a;q)_n=\frac{(a;q)_\infty}{(aq^n;q)_\infty}
\end{equation*}
for any integer $n$. Let
$$(a,b,\cdots,c;
q)_k=(a;q)_k(b;q)_k\cdots(c; q)_k, $$
 where $k$ is any integer or
$\infty$.
 The $(m+1)$-basic hypergeometric series $\Phi$
(see \cite[p. 95, Eq. (3.9.1) and (3.9.2)]{gasper}) and $\Psi$ are
defined respectively by
\begin{eqnarray*}\label{mphi}
\Phi \left(
\begin{array}{c}
  a_1,\cdots,a_{r}: c_{1,1},\cdots,c_{1,r_1}:\cdots: c_{m,1},\cdots,c_{m,r_m}\\ b_1,\cdots,b_{r-1}:d_{1,1},\cdots,d_{1,r_1}:\cdots: d_{m,1},\cdots,d_{m,r_m}
\end{array}
;q,q_1,\cdots,q_m;z
\right)\notag\\
=\sum\limits_{n=0}^{\infty}
\frac{(a_1,\cdots,a_{r};q)_n}{(q,b_1,\cdots,b_{r-1};q)_n}z^n\prod\limits_{j=1}^{m}\frac{(c_{j,1},\cdots,c_{j,r_j};q_j)_n}{(d_{j,1},\cdots,d_{j,r_j};q_j)_n}
\notag
\end{eqnarray*}
and
\begin{eqnarray*}\label{mpsi}
\Psi \left(
\begin{array}{c}
  a_1,\cdots,a_{r}: c_{1,1},\cdots,c_{1,r_1}:\cdots: c_{m,1},\cdots,c_{m,r_m}\\ b_1,\cdots,b_{r}:d_{1,1},\cdots,d_{1,r_1}:\cdots: d_{m,1},\cdots,d_{m,r_m}
\end{array}
;q,q_1,\cdots,q_m;z
\right)\notag\\
=\sum\limits_{n=-\infty}^{\infty}
\frac{(a_1,\cdots,a_{r};q)_n}{(b_1,\cdots,b_{r};q)_n}z^n\prod\limits_{j=1}^{m}\frac{(c_{j,1},\cdots,c_{j,r_j};q_j)_n}{(d_{j,1},\cdots,d_{j,r_j};q_j)_n}.
\notag
\end{eqnarray*}

The main results of this paper are Theorem \ref{1} and \ref{2}
below.

\begin{thm}\label{1} For $\left|{1\over st}\right|<1$, there holds
\begin{eqnarray*}\label{2phi}
\Phi \left(
\begin{array}{cccccc}
  a^2,&aq^2,&-aq^2&:&s,&t\\ &a,&-a&:&aq/s,&aq/t
\end{array}
;q^2,q;-{1\over st} \right)\\
=\frac{(s+t)}{st}\frac{(aq,-q/s,-q/t,aq/st;q)_\infty}{(-q,aq/s,aq/t,-1/st;q)_\infty}\notag.\notag
\end{eqnarray*}
\end{thm}
\begin{thm}\label{2} For $\left|{a^2\over bst}\right|<1$, there holds
\begin{eqnarray*}\label{2psi}
\lefteqn{{}\Psi \left(
\begin{array}{cccccc}
  aq^2,&-aq^2,&b&:&s,&t\\ a,&-a,&a^2q^2/b&:&aq/s,&aq/t
\end{array}
;q^2,q;-{a^2\over bst} \right)}\hspace{-0.2cm}
\\&=&\frac{a(s+t)}{(a+1)st}\frac{(q,q/a,aq,aq/st,-a/b;q)_\infty(a^2q^{2}/bs^2,a^2q^{2}/bt^2;q^2)_\infty}{
(q/s,q/t,aq/s,aq/t,-a^2/bst;q)_\infty(a^2q^2/b,q^{2}/b;q^2)_{\infty}}.\notag
\end{eqnarray*}
\end{thm}
For miscellaneous summation formulas of basic hypergeometric series,
the readers can consult Gasper and Rahman \cite{gasper}.

This paper is organized as follows.

In Section \ref{mainproof}, we will firstly prove Theorem \ref{1}
from a terminating summation formula in \cite{gasper}, and then,
Theorem \ref{2} will be deduced from Theorem \ref{1}. In both the
proofs, a method similar to that in \cite{ismail} and \cite{askism}
is used, where Ramanujan's $_1\psi_1$ and Bailey's $_6\psi_6$
summation formulas were proved respectively.

In Section \ref{qana},   we will prove that Theorem \ref{1} is a
$q$-analogue of a $_4F_3$-summation formula in Andrews, Askey and
Roy's book \cite{aar}.

In Section \ref{specia}, special cases of Theorem \ref{2} will be
considered. We will prove that Theorem \ref{2} is a generalization
of Jacobi's triple product identity and the quintuple product
identity.

In the following,  LHS (or RHS) means the left (or right) hand side
of a certain equality and $\textbf{N}$ denotes the set of
nonnegative integers.


\section{Proofs of  Theorem \ref{1} and \ref{2}  } \label{mainproof}
The lemma ( See, for example, \cite[p.90, Thm 1.2]{lang}) below is
the foundation of our proofs in this section.
\begin{lem}\label{thmfunde}
 Let $U$ be a connected open set and $f$, $g$
be analytic on $U$.   If $f$ and $g$ agree infinitely often near an
interior point of $U$, then we have $ f(z)=g(z)$ ~for all~$z\in U$.
\end{lem}
\begin{proof}[Proof of Theorem \ref{1}]
We begin with the identity \cite[p. 98, Eq. (3.10.5)]{gasper}:
\begin{eqnarray}\label{2phizz}\\
{}\Phi \left(
\begin{array}{ccc}
  a^2,aq^2,-aq^2&:&-aq/w,q^{-n}\\ a,-a&:&w,-aq^{n+1}
\end{array}
;q^2,q;{wq^{n-1}\over a}
\right)=\frac{(-aq,aq^2/w,w/aq;q)_n}{(-q,aq/w,w;q)_n},\notag\\
n\in\textbf{N}.            \notag
\end{eqnarray}
Rewrite \eqref{2phizz} to be
\begin{eqnarray} \label{lhsz1z} \\
\sum_{k=0}^{\infty}\frac{(a^2,aq^2,-aq^2;q^2)_k(-aq/w;q)_k(q^n-1)(q^n-q)\cdots(q^n-q^{k-1})}
{(q^2,a,-a;q^2)_k (w;q)_k(1+aq^{n+1})(1+aq^{n+2})\cdots(1+aq^{n+k})}
\left(\frac{w}{aq}\right)^k \notag\\
=\frac{(-aq,aq^2/w,w/aq,-q^{n+1},aq^{n+1}/w,wq^{n};q)_\infty}{(-q,aq/w,w,-aq^{n+1},aq^{n+2}/w,wq^{n-1}/a;q)_\infty}.
\notag
\end{eqnarray}
Set
\begin{eqnarray*} \label{}
f_1(z)=
\sum_{k=0}^{\infty}\frac{(a^2,aq^2,-aq^2;q^2)_k(-aq/w;q)_k(z-1)(z-q)\cdots(z-q^{k-1})}
{(q^2,a,-a;q^2)_k (w;q)_k(1+azq)(1+azq^{2})\cdots(1+azq^{k})}
\left(\frac{w}{aq}\right)^k  \notag
\end{eqnarray*}
and
\begin{eqnarray*} \label{}
f_2(z)=\frac{(-aq,aq^2/w,w/aq,-zq,azq/w,wz;q)_\infty}{(-q,aq/w,w,-azq,azq^{2}/w,wz/aq,;q)_\infty}.
\end{eqnarray*}
Then, \eqref{lhsz1z} shows
\begin{equation}\label{ff}f_1(z)=f_2(z)\end{equation}
 in $z=q^n$
with $n\in \textbf{N}$. According to Lemma \ref{thmfunde}, we have
\eqref{ff}
 for all $|z|<\min\left\{\frac{1}{|aq|},
\left|\frac{w}{aq^2}\right|, \left|\frac{aq}{w}\right| \right\}$.
By analytic continuation, the restriction on $z$ may be relaxed. Let
$z=-\frac{u}{aq}$ in \eqref{ff}. Then, simple computation gives
\begin{eqnarray}\label{2phizzz}
\Phi \left(
\begin{array}{cccccc}
  a^2,&aq^2,&-aq^2&:&-aq/w,&-aq/u\\ &a,&-a&:&w,&u
\end{array}
;q^2,q;-{wu\over a^2q^2}
\right)\\=-\frac{(u+w)}{aq}\frac{(-aq,w/a,u/a,-wu/aq;q)_\infty}{(-q,w,u,-wu/a^2q^2;q)_\infty}.\notag
\end{eqnarray}
 Replacing $a$ by $-a$ in \eqref{2phizzz}, and then, replacing $w$ by $\frac{aq}{s}$ and $u$ by
$\frac{aq}{t}$  respectively, we obtain Theorem \ref{1}. This
completes the proof.
\end{proof}

\begin{proof}[Proof of  Theorem \ref{2}] For $m\in \textbf{N}$, we have
\begin{eqnarray} \label{guocheng}
\lefteqn{\sum_{k=-m}^{\infty}\frac{(aq^2,-aq^2,a^2q^{-2m};q^2)_k(s,t;q)_k}
{(a,-a,q^{2+2m};q^2)_k (aq/s,aq/t;q)_k}
\left(-\frac{q^{2m}}{st}\right)^k} \\
&=&
\sum_{k=0}^{\infty}\frac{(aq^2,-aq^2,a^2q^{-2m};q^2)_{k-m}(s,t;q)_{k-m}}
{(a,-a,q^{2+2m};q^2)_{k-m} (aq/s,aq/t;q)_{k-m}}
\left(-\frac{q^{2m}}{st}\right)^{k-m}  \notag  \\
&=& \frac{(a^2q^{-2m};q^2)_{-m}(s,t;q)_{-m}} {(q^{2+2m};q^2)_{-m}
(aq/s,aq/t;q)_{-m}}
\left(-\frac{st}{q^{2m}}\right)^{m}\frac{(1-a^2q^{-4m})}{(1-a^2)} \notag\\
&&\times
\sum_{k=0}^{\infty}\frac{(1-a^2q^{4(k-m)})}{(1-a^2q^{-4m})}\frac{(a^2q^{-4m};q^2)_{k}(sq^{-m},tq^{-m};q)_{k}}
{(q^{2};q^2)_{k} (aq^{1-m}/s,aq^{1-m}/t;q)_{k}}
\left(-\frac{q^{2m}}{st}\right)^{k}.
  \notag
\end{eqnarray}
Apply  Theorem \ref{1} to summarize  the last sum above and note
that
\begin{equation*}
(a;q)_{-m}=\frac{1}{(aq^{-m};q)_{m}}=\frac{(-1)^mq^{m(m+1)/2}}{a^m(q/a;q)_m}.
\end{equation*}
Then,
\begin{eqnarray} \label{guochen}
\lefteqn{\mbox{RHS of
\eqref{guocheng}}}\hspace{-0.3cm}\\&=&\frac{(q^{2};q^2)_{m}(s/a,t/a;q)_{m}}
{(q^{2m+2}/a^2;q^2)_{m}
(q/s,q/t;q)_{m}}\frac{q^{m(m+3)}(1-a^2q^{-4m}) }{s^mt^m(1-a^2)}
\notag\\
&&\times\frac{(s+t)q^m}{st}\frac{(aq^{1-2m},-q^{1+m}/s,-q^{1+m}/t,-aq/st;q)_\infty}{(-q,aq^{1-m}/s,aq^{1-m}/t,-q^{2m}/st;q)_\infty}
\notag  \\
&=&\frac{(s+t)(a+q^{2m})(q^2;q^2)_m(q/a;q)_{2m}(aq,-q/s,-q/t,aq/st;q)_\infty}
{st(1+a)(q^{2m+2}/a^2,q^2/s^2,q^2/t^2;q^2)_{m}(-q,aq/s,aq/t,-q^{2m}/st;q)_\infty}
\frac{}{}\notag\\
&=&
\frac{a(s+t)}{(1+a)st}\frac{(q^{2m+2}/s^2,q^{2m+2}/t^2;q^2)_\infty(q,-q^{2m}/a,q/a;q)_{\infty}}{
(q^{2m+2},q^{2m+2}/a^2;q^2)_{\infty}(q/s,q/t,aq/s,aq/t;q)_{\infty}}
\notag
\\&&\times
\frac{(aq,aq/st;q)_\infty}{(-q^{2m}/st;q)_\infty}.\notag
\end{eqnarray}
In \eqref{guocheng} and \eqref{guochen}, we have proved that
\begin{eqnarray} \label{guot}
\\
\lefteqn{\sum_{k=0}^{\infty}\frac{(aq^2,-aq^2;q^2)_k(s,t;q)_k(q^{2m}-a^2)(q^{2m}-a^2q^2)\cdots(q^{2m}-a^2q^{2k-2})}
{(a,-a;q^2)_k
(aq/s,aq/t;q)_k(1-q^{2m+2})(1-q^{2m+4})\cdots(1-q^{2m+2k})}
\left(-\frac{1}{st}\right)^k\notag}\\
\lefteqn{+\sum_{k=1}^{\infty}\frac{(q^2/a,-q^2/a;q^2)_k(s/a,t/a;q)_k(q^{2m}-1)\cdots(q^{2m}-q^{2k-2})}
{(1/a,-1/a;q^2)_k
(q/s,q/t;q)_k(1-q^{2m+2}/a^2)\cdots(1-q^{2m+2k}/a^2)}
\left(-\frac{1}{st}\right)^k \notag} \\
&=&
\frac{a(s+t)}{(1+a)st}\frac{(q^{2m+2}/s^2,q^{2m+2}/t^2;q^2)_\infty(q,-q^{2m}/a,q/a,
aq,aq/st;q)_\infty}{(q^{2m+2},q^{2m+2}/a^2;q^2)_{\infty}(q/s,q/t,aq/s,aq/t,-q^{2m}/st;q)_\infty}.
\notag\notag
\end{eqnarray}
Put
\begin{eqnarray*} \label{g1}
g_1(x)
&=&\sum_{k=0}^{\infty}\frac{(aq^2,-aq^2;q^2)_k(s,t;q)_k(x-a^2)(x-a^2q^2)\cdots(x-a^2q^{2k-2})}
{(a,-a;q^2)_k (aq/s,aq/t;q)_k(1-xq^2)(1-xq^4)\cdots(1-xq^{2k})}
\left(-\frac{1}{st}\right)^k\notag\\
&&+\sum_{k=1}^{\infty}\frac{(q^2/a,-q^2/a;q^2)_k(s/a,t/a;q)_k(x-1)\cdots(x-q^{2k-2})}
{(1/a,-1/a;q^2)_k (q/s,q/t;q)_k(1-xq^2/a^2)\cdots(1-xq^{2k}/a^2)}
\left(-\frac{1}{st}\right)^k\notag
\end{eqnarray*} and
\begin{equation*} \label{g2}
g_2(x)=\frac{a(s+t)}{(a+1)st}\frac{(xq^{2}/s^2,xq^{2}/t^2;q^2)_\infty(-x/a,q,q/a,aq,aq/st;q)_\infty}{
(xq^2,xq^{2}/a^2;q^2)_{\infty}(q/s,q/t,aq/s,aq/t,-x/st;q)_\infty}.
\end{equation*}
Then \eqref{guot} shows that
\begin{equation}\label{g1g2}
g_1(x)=g_2(x)
\end{equation}
for $x=q^{2m}$ with $m\in \textbf{N}$.  According to Lemma
\ref{thmfunde}, \eqref{g1g2} holds for all
$|x|<\min\left\{\frac{1}{|q^2|}, \left|\frac{a}{q}\right|^2, |st|,
\left|\frac{a}{q}\right| \right\}$. By analytic continuation, the
restriction on $x$ may be relaxed. Rewrite \eqref{g1g2} as
\begin{eqnarray} \label{zuizhongq}
\lefteqn{\sum_{k=-\infty}^{\infty}\frac{(aq^2,-aq^2,a^2/x;q^2)_k(s,t;q)_k}
{(a,-a,xq^{2};q^2)_k (aq/s,aq/t;q)_k} \left(-\frac{x}{st}\right)^k}
\\
&=&\frac{a(s+t)}{(a+1)st}\frac{(xq^{2}/s^2,xq^{2}/t^2;q^2)_\infty(-x/a,q,q/a,aq,aq/st;q)_\infty}{
(xq^{2},xq^{2}/a^2;q^2)_{\infty}(q/s,q/t,aq/s,aq/t,-x/st;q)_\infty}.\notag
\end{eqnarray}
Taking $x=a^2/b$ in \eqref{zuizhongq}, we arrive at  Theorem
\ref{2}, which completes the proof.
\end{proof}
There are many terminating summation formulas in \cite[p. 96--100,
\textsection 3.10]{gasper}, but \cite[p. 98, Eq.(3.10.5)]{gasper},
i.e. \eqref{2phizz} in this paper, is the only one we can generalize
to be nonterminating, and then,  bilateral in this way. The
difficulty is that we can't find an open set near zero such that two
analytic functions agree infinitely often.

The LHS of the formulas in Theorem \ref{1} can also be rewritten as
a $_8\phi_7$ series, and,  the LHS of the formulas in Theorem
\ref{2} can also be rewritten as a $_8\psi_8$.

\section{From  Theorem \ref{1} to a $_4F_3$-summation
formula}\label{qana}

The hypergeometric series $_rF_s$ is defined by \cite[p. 62, Eq.
(2.1.2)]{aar}:
\begin{equation*}\label{rfs}
_rF_s\left(\begin{array}{c} a_1,\ldots, a_r\\b_1,\ldots,b_s
\end{array};x\right)=\sum_{n=0}^{\infty}\frac{(a_1)_n\ldots(a_r)_n}{(b_1)_n\ldots(b_s)_n}\frac{x^n}{n!},
\end{equation*}
where $(a)_n=a(a+1)\cdots(a+n-1)$ for $n>0$ and $(a)_0=1$. The
$\Gamma$-function and its $q$-analogue, say, the
$\Gamma_q$-function, are defined, respectively,  by:
\begin{equation}\label{ggq}
\Gamma(x)=\lim_{n\rightarrow\infty}\frac{n!n^{x-1}}{(x)_n}
\mbox{\qquad and\qquad}
\Gamma_q(x)=\frac{(q;q)_\infty}{(q^x;q)_\infty}(1-q)^{1-x}.
\end{equation}
The definitions and properties of the $\Gamma$- and
$\Gamma_q$-functions can be found in the books \cite{aar} and
\cite{gasper}.  Note that
$\lim_{q\rightarrow1^{-}}\Gamma_q(x)=\Gamma(x)$.

We state the $_4F_3$-summation formula in the following.
\begin{cor}\cite[p. 148, Cor.
3.5.3]{aar}\label{aaa}
\begin{equation*}\label{4f3}
_4F_3\left(\begin{array}{cccc}
a,&{a\over2}+1,&c,&d\\&{a\over2},&a-c+1,&a-d+1
\end{array};-1\right)=\frac{\Gamma{(a-c+1)}\Gamma{(a-d+1)}}{\Gamma{(a+1)}\Gamma{(a-c-d+1)}}.
\end{equation*}
\end{cor}
\begin{proof} Performing the replacement: $a\rightarrow q^a$,  $s\rightarrow
q^c$ and $t\rightarrow q^d$, in Theorem \ref{1}, gives
\begin{eqnarray}\label{2phiq}
\sum_{k=0}^\infty\frac{(q^{2a},q^{a+2},-q^{a+2};q^2)_k(q^c,q^d;q)_k}{(q^2,q^a,-q^a;q^2)_k(q^{a-c+1},q^{a-d+1};q)_k}(-q^{-(c+d)})^k\\
=\frac{(q^c+q^d)}{q^{c+d}}\frac{(q^{1+a},-q^{1-c},-q^{1-d},q^{a-c-d+1};q)_\infty}{(-q,q^{a-c+1},q^{a-d+1},-q^{-c-d};q)_\infty}\notag,
\end{eqnarray}
where we temporarily assume $\rm {Re}(c+d)<0$ for the convergence of
the series. Using the definitions in \eqref{ggq}, we have
\begin{eqnarray}\label{2piq}
&&\frac{(q^c+q^d)}{q^{c+d}}\frac{(q^{1+a},-q^{1-c},-q^{1-d},q^{a-c-d+1};q)_\infty}{(-q,q^{a-c+1},q^{a-d+1},-q^{-c-d};q)_\infty}\\
&&=\frac{(q^c+q^d)}{q^{c+d}}\frac{(q,q^{1+a},q^{a-c-d+1},q^{-c-d};q)_\infty(q^{2(1-c)},q^{2(1-d)};q^2)_\infty}
{(q^{a-c+1},q^{a-d+1},q^{1-c},q^{1-d};q)_\infty(q^2,q^{2(-c-d)};q^2)_\infty}\notag\\
&&=\frac{(q^c+q^d)(1-q)}{q^{c+d}(1-q
^2)}\frac{\Gamma_q(a-c+1)\Gamma_q(a-d+1)\Gamma_q(1-c)\Gamma_q(1-d)\Gamma_{q^2}(-c-d)}
{\Gamma_q(a+1)\Gamma_q(a-c-d+1)\Gamma_q(-c-d)\Gamma_{q^2}(1-c)\Gamma_{q^2}(1-d)}\notag\\
&&\rightarrow\frac{\Gamma(a-c+1)\Gamma(a-d+1)\Gamma(1-c)\Gamma(1-d)\Gamma(-c-d)}
{\Gamma(a+1)\Gamma(a-c-d+1)\Gamma(-c-d)\Gamma(1-c)\Gamma(1-d)}\quad(q\rightarrow1^-)\notag\\
&&\phantom{\rightarrow}=\frac{\Gamma(a-c+1)\Gamma(a-d+1)}
{\Gamma(a+1)\Gamma(a-c-d+1)}\notag.
\end{eqnarray}
On the other hand, we have, for $k\in\textbf{N}$,
\begin{eqnarray}\label{bh}
&&\frac{(q^{2a},q^{a+2},-q^{a+2};q^2)_k(q^c,q^d;q)_k}{(q^2,q^a,-q^a;q^2)_k(q^{a-c+1},q^{a-d+1};q)_k}\\
&&\rightarrow\frac{2^k(a)_k4^k({a\over2}+1)_k(c)_k(d)_k}{2^k(k!)4^k({a\over2})_k(a-c+1)_k(a-d+1)_k}\quad(q\rightarrow1^-)\notag\\&&
\phantom{\rightarrow}=\frac{(a)_k({a\over2}+1)_k(c)_k(d)_k}{k!({a\over2})_k(a-c+1)_k(a-d+1)_k}.\notag
\end{eqnarray}
Letting $q\rightarrow1^-$ on both sides of \eqref{2phiq}, and
combining with \eqref{2piq} and \eqref{bh}, we obtain
\begin{eqnarray}\label{2phq}
\sum_{k=0}^\infty\frac{(a)_k({a\over2}+1)_k(c)_k(d)_k}{({a\over2})_k(a-c+1)_k(a-d+1)_k}\frac{(-1)^k}{k!}
=\frac{\Gamma(a-c+1)\Gamma(a-d+1)}
{\Gamma(a+1)\Gamma(a-c-d+1)}.\notag
\end{eqnarray}
This is the identity in Corollary \ref{aaa}. By analytic
continuation, the restriction $\rm {Re}(c+d)<0$ may be relaxed. This
completes the proof.
\end{proof}
For another  $q$-analogue of the formula in Corollary \ref{aaa}, we
see the VWP-balanced $_6\phi_5$ summation formula \cite[p. 44, Eq.
(2.7.1)]{gasper}:
\begin{eqnarray*} \label{ongq}
\sum_{k=0}^{\infty}\frac{(a, qa^{1\over2},-qa^{1\over2},b,c,d;q)_k}
{(q,a^{1\over2},-a^{1\over2},aq/b,aq/c,aq/d;q)_k }
\left(\frac{aq}{bcd}\right)^k
=\frac{(aq,aq/bc,aq/bd,aq/cd;q)_\infty}{
(aq/b,aq/c,aq/d,aq/bcd;q)_\infty},\notag
\end{eqnarray*}
where $\left|{aq\over bcd}\right|<1$. Let $b\rightarrow\infty$ to
get
\begin{eqnarray} \label{onq}
\sum_{k=0}^{\infty}\frac{(a, qa^{1\over2},-qa^{1\over2},c,d;q)_k}
{(q,a^{1\over2},-a^{1\over2},aq/c,aq/d;q)_k }q^{k(k-1)\over2}
\left(-\frac{aq}{cd}\right)^k =\frac{(aq,aq/cd;q)_\infty}{
(aq/c,aq/d;q)_\infty}.
\end{eqnarray}
In a similar way as that of the proof above, we can verify that
\eqref{onq}  is also a $q$-analogue of the formula in Corollary
\ref{aaa}.

\section{Some special cases of Theorem \ref{2}} \label{specia}
Letting $b\rightarrow\infty$ in Theorem \ref{2}, we have
\begin{prop}\label{four3}There holds
\begin{equation*}\label{s1}
\sum_{k=-\infty}^{\infty}\left(1-a^2q^{4k}\right)\frac{(s,t;q)_k} {
(aq/s,aq/t;q)_k}q^{k^2-k}\left(\frac{a^2}{st}\right)^k=\frac{a(s+t)}{st}\frac{(q,q/a,a,aq/st;q)_\infty}{
(q/s,q/t,aq/s,aq/t;q)_\infty}.
\end{equation*}
\end{prop}
Letting $t\rightarrow\infty$ in Propsotion \ref{four3} gives
Corollary \ref{utq}, which is a unification of Jacobi's triple
product identity and the quintuple product identity.
\begin{cor}\label{utq}
There holds
\begin{equation*}\label{tq}
\sum_{k=-\infty}^{\infty}\left(1-a^2q^{4k}\right)\frac{(s;q)_k} {
(aq/s;q)_k}q^{3k^2-3k\over2}\left(-\frac{a^2}{s}\right)^k=\frac{a(q,q/a,a;q)_\infty}{
s(q/s,aq/s;q)_\infty}.
\end{equation*}
\end{cor}
Now we deduce the quintuple product identity and Jacobi's triple
product identity, respectively,  from Corollary  \ref{utq} in the
following.
\begin{cor}\label{qunn}
{\rm(Quintuple product identity \cite[p. 82]{ben})}
 For $x\neq0$, there holds
\begin{equation*}\label{ax}
\sum_{k=-\infty}^\infty(-1)^kq^{k(3k-1)\over2}x^{3k}(1+xq^k)
=\frac{(
 q,
 q/x^2,x^2;q)_\infty}{(x,q/x
;q)_\infty}.
\end{equation*}
\end{cor}
\begin{proof}
Replace $a$ by $x^2q$ and $s$ by $xq$  in Corollary \ref{tq},
respectively. After routine simplification, we have
\begin{equation}\label{tqg}
\sum_{k=-\infty}^{\infty}(-1)^k\left(1-x^4q^{4k+2}\right)q^{3k^2-k\over2}x^{3k}=\frac{(
 q,
 q/x^2,x^2;q)_\infty}{(x,q/x
;q)_\infty}.
\end{equation}
The LHS of \eqref{tqg} is equal to
\begin{eqnarray}\label{tg}
&&\sum_{k=-\infty}^{\infty}(-1)^kq^{3k^2-k\over2}x^{3k}+\sum_{k=-\infty}^{\infty}(-1)^{k
+1}q^{3k^2+7k+4\over2}x^{3k+4}\\
&&=\sum_{k=-\infty}^{\infty}(-1)^kq^{3k^2-k\over2}x^{3k}+\sum_{k=-\infty}^{\infty}(-1)^{k
}q^{3k^2+k\over2}x^{3k+1}\notag\\
&&=\sum_{k=-\infty}^\infty(-1)^kq^{k(3k-1)\over2}x^{3k}(1+xq^k).\notag
\end{eqnarray}
Combining \eqref{tqg} and \eqref{tg}, we obtain Corollary
\ref{qunn}. This completes the proof.
\end{proof}

\begin{cor}\label{trii}
{\rm(Jacobi's triple product identity \cite[p. 497]{aar}, \cite[p.
35]{ben} or \cite[p. 15, Eq. (1.6.1)]{gasper})}
 For $a\neq0$, there holds
\begin{equation*}\label{jcobi}
\sum_{n=-\infty}^{\infty}(-1)^nq^{n(n-1)\over2}a^n=(q,q/a,a;q)_\infty.
\end{equation*}
\end{cor}
\begin{proof} Rewrite the LHS of the formula in Corollary \ref{tq} as
\begin{eqnarray}\label{tnq}
&&\sum_{k=-\infty}^{\infty}\frac{(s;q)_k} {
(aq/s;q)_k}q^{3k^2-3k\over2}\left(-\frac{a^2}{s}\right)^k-a^2\sum_{k=-\infty}^{\infty}\frac{(s;q)_k}
{ (aq/s;q)_k}q^{3k^2+5k\over2}\left(-\frac{a^2}{s}\right)^k\\
&&=\sum_{k=-\infty}^{\infty}\frac{(s;q)_k} {
(aq/s;q)_k}q^{3k^2-3k\over2}\left(-\frac{a^2}{s}\right)^k-\sum_{k=-\infty}^{\infty}\frac{(s;q)_{k-1}}
{ (aq/s;q)_{k-1}}q^{3k^2-k-2\over2}\frac{(-1)^{k-1}a^{2k}}{s^{k-1}}\notag\\
&&=\sum_{k=-\infty}^{\infty}\frac{(s;q)_{k-1}} {
(aq/s;q)_{k-1}}q^{3k^2-3k\over2}\left(-\frac{a^2}{s}\right)^{k-1}\left[{1-sq^{k-1}\over1-aq^k/s}\left(-{a^2\over
s}\right)-q^{k-1}a^2\right]\notag\\
&&=\sum_{k=-\infty}^{\infty}\frac{(s;q)_{k-1}} {
(aq/s;q)_{k-1}}q^{3k^2-3k\over2}\left(-\frac{a^2}{s}\right)^{k-1}{-a^2+a^3q^{2k-1}\over
s-aq^k}.\notag
\end{eqnarray}
Multiply both sides of  the formula in Corollary \ref{tq} by $s\over
a$, and use \eqref{tnq} to get
\begin{equation}\label{tqs}
\sum_{k=-\infty}^{\infty}\frac{(s;q)_{k-1}} {
(aq/s;q)_{k-1}}q^{3k^2-3k\over2}\left(-\frac{a^2}{s}\right)^{k-1}{s(-a+a^2q^{2k-1})\over
s-aq^k}=\frac{(q,q/a,a;q)_\infty}{(q/s,aq/s;q)_\infty}.
\end{equation}
We look both sides of \eqref{tqs} as functions of $s$.   Note that
the point $s=\infty$ is their removable pole, and then, the series
on the LHS of \eqref{tqs} is convergent uniformly in a neighborhood
of $\infty$. Let $s\rightarrow\infty$. The sum and the limit can
interchange. Thus, we have
\begin{equation}\label{lim}
\sum_{k=-\infty}^{\infty}\lim_{s\rightarrow\infty}\left(\frac{(s;q)_{k-1}}
{
(aq/s;q)_{k-1}}q^{3k^2-3k\over2}\left(-\frac{a^2}{s}\right)^{k-1}{s(-a+a^2q^{2k-1})\over
s-aq^k}\right)=(q,q/a,a;q)_\infty.
\end{equation}
Now we compute the limitation behind the sum above. For any integer
$k$,
\begin{eqnarray}\label{lim2}
&&\lim_{s\rightarrow\infty}\left(\frac{(s;q)_{k-1}} {
(aq/s;q)_{k-1}}q^{3k^2-3k\over2}\left(-\frac{a^2}{s}\right)^{k-1}{s(-a+a^2q^{2k-1})\over
s-aq^k}\right)\\
&&=q^{3k^2-3k\over2}a^{2k-2}\lim_{s\rightarrow\infty}\left(\frac{(s;q)_{k-1}}
{
(aq/s;q)_{k-1}}\left(-\frac{1}{s}\right)^{k-1}\right)\lim_{s\rightarrow\infty}\left({s(-a+a^2q^{2k-1})\over
s-aq^k}\right)\notag\\
&&=q^{3k^2-3k\over2}a^{2k-2}\times
q^{k^2-3k+2\over2}\times\left(-a+a^2q^{2k-1}\right)\notag\\
&&=-q^{2k^2-3k+1}a^{2k-1}+q^{2k^2-k}a^{2k}.\notag
\end{eqnarray}
In \eqref{lim} and \eqref{lim2}, we have proved
\begin{eqnarray}\label{lm2}
-\sum_{k=-\infty}^{\infty}q^{2k^2-3k+1}a^{2k-1}+\sum_{k=-\infty}^{\infty}q^{2k^2-k}a^{2k}=(q,q/a,a;q)_\infty.
\end{eqnarray}
Combining the two sums on the LHS of \eqref{lm2}, we arrive at
Corollary \ref{trii}. This ends the proof.
\end{proof}

\end{document}